\documentclass[12]{amsart}%
\usepackage{hyperref}
\usepackage{xcolor}
\usepackage{amsmath}
\usepackage{amsfonts}
\usepackage{amssymb}
\usepackage{graphicx}%
\providecommand{\U}[1]{\protect\rule{.1in}{.1in}}
\newtheorem{thm}{Theorem}[section]

\newtheorem{lem}[thm]{Lemma}
\newtheorem{corol}[thm]{Corollary}
\theoremstyle{definition}

\newcommand{\forma}[1]{\langle #1 \rangle}

\newcommand{\norma}[1]{\Vert #1 \Vert}

\email{fidelisb@gmail.com}
\email{danielfbustosrios@gmail.com}
\email{esidney@gmail.com}
\email{fuspedro@gmail.com}
\email{jaime.ripoll@ufrgs.br}
\begin{document}
\title[ ]{MINIMAL ISOPARAMETRIC SUBMANIFOLDS OF $\mathbb{S}^{7}$ AND OCTONIONIC EIGENMAPS}
\author[ ]{Fidelis Bittencourt }
\address{Departamento de matem\'atica, Universidade Federal de Santa Maria, Santa
Maria, Brazil}
\author[ ]{Daniel Bustos R.}
\address{Instituto de matem\'atica e estat\'istica, Universidade Federal do Rio Grande
do Sul, Porto Alegre, Brazil}
\author[ ]{Edson S. Figueiredo}
\address{Departamento de matem\'atica, Universidade Federal de Santa Maria, Santa
Maria, Brazil}
\author[ ]{Pedro Fusieger}
\address{Departamento de matem\'atica, Universidade Federal de Santa Maria, Santa
Maria, Brazil}
\author[ ]{Jaime B. Ripoll}
\address{Instituto de matem\'atica e estat\'istica, Universidade Federal do Rio Grande
do Sul, Porto Alegre, Brazil}

\begin{abstract}
{\tiny We use the octonionic multiplication $\cdot$ of $\mathbb{S}^{7}$ to associate,
to each unit normal section $\eta$ of a submanifold $M$ of $\mathbb{S}^{7},$
an octonionic Gauss map $\gamma_{\eta}:M\rightarrow\mathbb{S}^{6},$
$\gamma_{\eta}(x)=x^{-1}\cdot\eta(x),$ $x\in M,$ where
$\mathbb{S}^{6}$ is the unit sphere of $T_{1}\mathbb{S}^{7},$
$1$ is the neutral element of $\cdot$ in $\mathbb{S}^{7}.$
Denoting by $\mathcal{N}(M)$ the vector bundle of normal sections of $M$ we
set, for $\eta$ $\in\mathcal{N}(M),$ $S_{\eta}(X)=\textcolor{red}{-}\left(
\nabla_{X}\eta\right)  ^{\top},$ $X\in TM.$ The
Hilbert-Schmidt inner product $\langle S_{\eta},S_{\nu
}\rangle$ on the vector bundle
\[
\mathcal{S}(M)=\left\{  S_{\eta}\text{
$\vert$
}\eta\in\mathcal{N}(M)\right\}
\]
is the trace of the bilinear form
\[
(X,Y)\in TM\times TM\mapsto\langle S_{\eta}(X),S_{\nu}(Y)\rangle\in
\mathbb{R}.
\]
Defining the bundle map $\mathcal{B}:\mathcal{N}(M)\rightarrow\mathcal{S}(M)$
by $\mathcal{B}(\eta)=S_{\eta},$ we prove that if $M$ is a minimal submanifold
of $\mathbb{S}^{7}$ and $\eta\in\mathcal{N}(M)$ is unitary and  parallel on
the normal connection, then $\gamma_{\eta}$ is harmonic if and only if $\eta$
is an eigenve\textcolor{red}{c}tor of
\[
\mathcal{B}^{\ast}\mathcal{B}:\mathcal{N}(M)\rightarrow\mathcal{N}%
(M)\textcolor{red}{,}
\]
that is, there is $\lambda\in C^{\infty}\left(  M\right)  $ such that
$\mathcal{B}^{\ast}\mathcal{B}\left(  \eta\right)  =\lambda\eta,$ where
$\mathcal{B}^{\ast}$ is the adjoint of $\mathcal{B}.$ If $M$ is an
isoparametric compact minimal submanifold of codimension $k$ of $\mathbb{S}%
^{7}$ then $\mathcal{B}^{\ast}\mathcal{B}$ has constant non negative
eigenvalues $0\leq\sigma_{1}\leq\cdots\leq\sigma_{k}$ and the associated
eigenvectors $\eta_{1},\cdots,\eta_{k}$ form an orthonormal basis of
$\mathcal{N}(M)$, parallel on the normal connection, such that each
$\gamma_{\eta_{j}}$ is an eigenmap of $M$ with eigenvalue $7-k+$ $\sigma
_{j},$ that is, $\Delta\gamma_{\eta_{j}}=-\left(
7-k+\sigma_{j}\right)  \gamma_{\eta_{j}}.$ Moreover, $\sigma_{j}=\Vert
S_{\eta_{j}}\Vert^{2},$ $1\leq j\leq k.$ It follows that each function
$\left\langle \gamma_{n_{j}},e\right\rangle $ is an eigenfunction for the
Laplacian of $M$ with eigenvalue $7-k+\sigma_{j},$ $1\leq j\leq k,$ for any
given $e\in T_{1}\mathbb{S}^{7}.$ Considering $\mathbb{S}^{m}$ as a totally
geodesic submanifold of $\mathbb{S}^{7}$, $3\leq m\leq7,$ if $M^{m-1}$ is a
minimal hypersurface of $\mathbb{S}^{m}$ and $\eta$ is an unit normal vector
field to $M$ in $\mathbb{S}^{m}$ then
$\gamma_{\eta}$ is a harmonic map. If $M$ is a compact, minimal submanifold of
$\mathbb{S}^{7}$ and $\eta$ an unit normal eigenvector of 
$\mathcal{B}^{\ast}\mathcal{B}$ then the Gauss image $\gamma_{\eta}(M)$
is not contained in an open hemisphere of $\mathbb{S}^{6}.$ }

\end{abstract}
\maketitle

\allowdisplaybreaks

\section{Introduction}
There is a vast literature extending and studying, under several point of views, the Gauss map of surfaces of the Euclidean space to submanifolds of arbitrary dimension and codimension and to more general ambient spaces. This study comprises notably  minimal and parallel mean curvature vector submanifolds (some well known and representative references, which are a fraction of what have already been done are: \cite{BN}, \cite{CL}, \cite{HO}, \cite{I}, \cite{O}, \cite{Os}). 

In \cite{BFLR} the authors use the octonionic product $\cdot$ of
$\mathbb{S}^{7}$ to define a Gauss map $\gamma_{\eta}:M\rightarrow
\mathbb{S}^{6}\subset T_{1}\mathbb{S}^{7}$of an orientable hypersurface $M$ of
$\mathbb{S}^{7}$ by
\begin{equation}
\gamma_{\eta}(x)=x^{-1}\cdot\eta(x),x\in M, \label{g}%
\end{equation}
where $\eta$ is an unit normal vector field of $M,$ $1$ is the neutral element
of $\cdot,$ $\mathbb{S}^{6}$ is the unit sphere of $T_{1}\mathbb{S}^{7}.$ They
prove that $M$ has constant mean curvature if and only if $\gamma_{\eta}$ is
harmonic and use this characterization of a CMC hypersurface of $\mathbb{S}%
^{7}$ to describe the geometry and topology of $M$ under conditions on the
Gauss image $\gamma_{\eta}\left(  M\right)  .$ We use here the octonionic
structure of $\mathbb{S}^{7}$ to study the Gauss map determined by unit normal
sections of minimal submanifolds of arbitrary codimension of $\mathbb{S}^{7}$.
Our main application consist in presenting explicit eigenmaps of minimal
isoparametric submanifolds of $\mathbb{S}^{7}.$ To state our main results we
have to introduce some notations and definitions.

Let $M$ be a submanifold of $\mathbb{S}^{7}.$ We denote by $\mathcal{N}\left(
M\right)  $ the vector bundle over $M$ of the normal sections of the normal
bundle
\[
\left\{  \left(  x,\eta\right)  \in T\mathbb{S}^{7}\text{
$\vert$
}\eta\in\left(  T_{x}M\right)  ^{\bot}\right\}
\]
of $M.$ Given $\eta\in\mathcal{N}\left(  M\right)  $ we set $S_{\eta
}(X)=-\left(  \nabla_{X}\eta\right)  ^{\top},$ where $\nabla$ is the
Riemannian connection on $\mathbb{S}^{7},$ $X$ a tangent vector field to $M$
and $\top$ the orthogonal projection on $TM.$
We denote by $\mathcal{S}\left(  M\right)  $ the vector bundle over $M$ of non
normalized second fundamental forms of $M,$ namely:%
\[
\mathcal{S}\left(  M\right)  =\left\{  S_{\eta}\text{
$\vert$
}\eta\in\mathcal{N}\left(  M\right)  \right\}  .
\]
The Hilbert-Schmidt inner product $\langle S_{\eta},S_{\nu
}\rangle$ on $\mathcal{S}\left(  M\right)  $ is defined as the trace of the
bilinear form
\[
(X,Y)\in TM\times TM\mapsto\langle S_{\eta}(X),S_{\nu}(Y)\rangle\in
\mathbb{R}.
\]
Define a bundle map $\mathcal{B}:\mathcal{N}\left(  M\right)  \rightarrow
\mathcal{S}\left(  M\right)  $ by $\mathcal{B}\left(  \eta\right)  =S_{\eta}$
and by $\mathcal{B}^{\ast}:\mathcal{S}\left(  M\right)  \rightarrow
\mathcal{N}\left(  M\right)  $ the adjoint of $\mathcal{B}$.

We may see that a smooth map $\gamma:M\rightarrow\mathbb{S}^{6}\subset
\mathbb{R}^{7}$ is harmonic, that is, a critical point of the
functional
\[
g\mapsto\int_{M}\left\Vert Dg\right\Vert ^{2},
\]
$g:M\rightarrow\mathbb{S}^{6}$ smooth with compact support, if and only if
\begin{equation}
\Delta\gamma=\lambda\gamma\label{ei}%
\end{equation}
for some function $\lambda$ on $M$, where
\[
\Delta\gamma=\sum_{i=1}^{7}\left(  \Delta_{M}\langle\gamma,e_{i}%
\rangle\right)  e_{i},
\]
$\Delta_{M}$ is the usual Laplacian on $M$ and $\left\{  e_{i}\right\}  $ any
fixed orthonormal basis of $\mathbb{R}^{7}.$ Using the octonionic structure of
$\mathbb{S}^{7}$ to define $\gamma_{\eta}:M\rightarrow\mathbb{S}^{6}$ by
(\ref{g}) we prove:

\begin{thm}
\label{mink} Let $M$ be a minimal submanifold of codimension
$1\leq k\leq6$ of $\mathbb{S}^{7}$ and let $\eta\in\mathcal{N}\left(
M\right)  $ be an unit normal section, parallel on the normal
connection of $M\textcolor{red}{,}$ that is, $\left(  \nabla_{X}\eta\right)
^{\bot}=0$ for all $X\in TM.$ Then the following alternatives are equivalent:

\qquad(i) $\gamma_{\eta}:M\rightarrow\mathbb{S}^{6}\subset T_{1}\mathbb{S}%
^{7}$ satisfies%
\[
\Delta\gamma_{\eta}=-\left(  7-k+\left\Vert S_{\eta
}\right\Vert ^{2}\right)  \gamma_{\eta}\textcolor{red}{.}
\]

\bigskip

\qquad(ii) $\eta$ is an eigenvector of $\mathcal{B}^{\ast}\mathcal{B}$ with
eigenvalue $\left\Vert S_{\eta}\right\Vert ^{2}\textcolor{red}{.}$

\bigskip

\qquad(iii) $\gamma_{\eta}:M\rightarrow\mathbb{S}^{6}$ is harmonic.
\end{thm}

A straightforward consequence of the  Theorem  \ref{mink} is:

\begin{corol}
\label{min1} Consider $\mathbb{S}^{m}$ as a totally geodesic submanifold of
$\mathbb{S}^{7},$ $3\leq m\leq7.$ Let $M^{m-1}$ be an oriented minimal
hypersurface of $\mathbb{S}^{m}$, and let $\eta$ be an unit normal vector
field to $M$ in $\mathbb{S}^{m}$. Then,
$\gamma_{\eta}$ is a harmonic map.
\end{corol}

Recall that a submanifold $M$ of $\mathbb{S}^{n}$ is called isoparametric if
it has flat normal bundle (zero normal curvature) and the principal curvatures
along any parallel normal field are constant. It is known that any
isoparametric submanifold of $\mathbb{S}^{n}$ is a leaf of a foliation
(singular) of $\mathbb{S}^{n}$ by isoparametric submanifolds and that this
foliations contains a leaf which is regular and minimal (see \cite{T} and
\cite{PT}, Section 6). We also recall that $\gamma:M\rightarrow\mathbb{S}^{6}$
is an eigenmap if it is harmonic and the function $\lambda$ in (\ref{ei}) is
constant (see also \cite{QT}, \cite{EL}). We prove

\begin{thm}
\label{isoparametric} If $M$ is an isoparametric compact minimal submanifold
of codimension $1\leq k\leq 6$ of $\mathbb{S}^{7}$ then $\mathcal{B}^{\ast}\mathcal{B}$
has constant non negative eigenvalues $0\leq\sigma_{1}\leq\cdots\leq\sigma
_{k}$ and the associated eigenvectors $\eta_{1},\cdots,\eta_{k}$ form an
orthonormal basis of $\mathcal{N}(M)$, parallel on the normal connection, such
that each $\gamma_{\eta_{j}}$ is an eigenmap of $M$ with eigenvalue $7-k+$
$\sigma_{j}\textcolor{red}{,}$ that is, $\Delta\gamma_{\eta_{j}}=-\left(
7-k+\sigma_{j}\right)  \gamma_{\eta_{j}}.$ Moreover, $\sigma_{j}=\Vert
S_{\eta_{j}}\Vert^{2},$ $1\leq j\leq k.$ It follows that each function
$\left\langle \gamma_{n_{j}},e\right\rangle $ is an eigenfunction for the
Laplacian of $M$ with eigenvalue $7-k+\sigma_{j},$ $1\leq j\leq k,$ for any
given $e\in T_{1}\mathbb{S}^{7}.$
\end{thm}

The image of the Gauss map of a minimal surface in the Euclidean space is a
classical topic of study in Differential Geometry and there is a vast
literature on this subject. It is well known that if the image of the Gauss
map of a complete minimal surface of $\mathbb{R}^{3}$ is contained in a
hemisphere of $\mathbb{S}^{2}$ then the surface is a plane. By defining a
Gauss map of an orientable minimal hypersurface $M^{n}$ of the sphere
$\mathbb{S}^{n+1}$ with an unit normal vector field $\eta$ as the usual
Euclidean Gauss map $\gamma:M\rightarrow\mathbb{S}^{n+1},$ $\gamma
(x)=\eta(x),$ $x\in M,$ E. De Giorgi \cite{DG} and, independently J. Simons
\cite{JS} proved that if $M$ is compact and $\gamma(M)$ lies in an open
hemisphere of $\mathbb{S}^{n+1}$ then $M$ must be a great hypersphere in
$\mathbb{S}^{n+1}.$ As a consequence of Theorems \ref{mink} and \ref{min1}, we
obtain here a similar result for minimal submanifolds of arbitrary codimension
of $\mathbb{S}^{7}$:

\begin{thm}
\label{image}
 Let $M$ be a compact and minimal submanifold of codimension $1\leq k \leq 5$ of the sphere $\mathbb{S}^7.$ Let $\eta$ be an unit normal vector parallel in the normal bundle of $M$ that is  an eigenvector of
$\mathcal{B}^{\ast}\mathcal{B}$. Then the image
of the octonionic Gauss map $\gamma_{\eta}$ is not contained in an open
hemisphere of $\mathbb{S}^{6}.$
\end{thm}

\section{The octonionic structure of $\mathbb{S}^{7}$ and the octonionic Gauss
map.}

The octonions is a $8-$dimensional Cayley-Dickson algebra $\mathcal{C}_{8}.$
Given a number $n\in\{0,1,2,\dots\},$ the Cayley-Dickson algebra
$\mathcal{C}_{n}$ is a division algebra structure on $\mathbb{R}^{2^{n}}$
defined inductively by $\mathcal{C}_{0}=\mathbb{R}$ and by the following
formulae: If $x=\left(  x_{1},x_{2}\right)  $, $y=\left(  y_{1},y_{2}\right)
$ are in $\mathbb{R}^{2^{n}}=\mathbb{R}^{2^{n-1}}\times\mathbb{R}^{2^{n-1}}$,
$n\geq1$, then
\begin{equation}
x\cdot y=\left(  x_{1}y_{1}-\overline{y_{2}}x_{2},y_{2}x_{1}+x_{2}%
\overline{y_{1}}\right)  , \label{xy}%
\end{equation}
where
\[
\overline{x}=\left(  \overline{x}_{1},-x_{2}\right)  ,
\]
with $\overline{x}=x$ if $x\in\mathbb{R}$ (see \cite{B}).

We use the notation $\mathbb{O}=\mathcal{C}_{8}$ for the octonions and denote
by $1$ the neutral element of $\mathbb{O}.$ We mention below some well known
facts about the octonions which proofs can be found in \cite{B}. Besides being
a division algebra, $\mathbb{O}$ is normed: $\Vert x\cdot y \Vert=\Vert x
\Vert\Vert y \Vert,$ for any $x,\ y\in\mathbb{O}$, where $\Vert\ \Vert$ is the
usual norm of $\mathbb{R}^{8},$ and $\Vert x \Vert=\sqrt{x\cdot\overline{x}}.$
Setting $\operatorname*{Re}(x)=\left(  x+\overline{x}\right)  /2$ we have
\[
T_{1}\mathbb{S}^{7}=\{x\in\mathbb{R}^{8}\ | \ \operatorname*{Re}(x)=0\}.
\]

The right and left translations $R_{x}, L_{x}:\mathbb{O}\rightarrow
\mathbb{O},$ $R_{x}(v)=v\cdot x,$ $L_{x}(v)=x\cdot v,$ $v\in\mathbb{O},$ are
orthogonal maps if $\Vert x \Vert=1$ and are skew-symmetric if
$\operatorname*{Re}(x)=0.$ In particular, the unit sphere $\mathbb{S}^{7}$ is
preserved by left and right translation of unit vectors and, moreover, any
$v\in T_{1}\mathbb{S}^{7}$ determines a Killing vector field $V$ of
$\mathbb{S}^{7}$ given by the left translation, $V(x)=x\cdot v$,
$x\in\mathbb{S}^{7}.$

Define
\begin{align*}
\Gamma: T\mathbb{S}^{7}  &  \to T_{1}\mathbb{S}^{7}\\
(x,v)  &  \mapsto L_{x^{-1}}(v).
\end{align*}
We shall also use the notation $\Gamma_{x} (v)=\Gamma(x,v).$ If $M$ is a
submanifold of $\mathbb{S}^{7},$ a global unit normal section $\eta$ of $M$
determines a octonionic Gauss map
\[
\gamma_{\eta}:M\rightarrow\mathbb{S}^{6} \subset T_{1}\mathbb{S}^{7}%
\]
by setting
\[
\gamma_{\eta}(x)=\Gamma_{x}(\eta(x))=x^{-1}\cdot\eta(x), \ \ x\in\mathbb{S}^{7},
\]
where $\mathbb{S}^{6}$ is the unit sphere of $T_{1}\mathbb{S}^{7}.$

\section{Proof of the results}

The following lemma is a basic but fundamental result of the paper:

\begin{lem}
\label{laplacian} Let $M$ be a $n$-dimensional minimal submanifold of
$\mathbb{S}^{7}$ and let $\eta$ be an unit normal vector field
and  parallel in the
normal connection of $M.$ Then, for $v\in\mathbb{S}^{6}\subset T_{1}%
\mathbb{S}^{7},$ setting
\[
f(x)=\langle\gamma_{\eta}(x),v \rangle,\ x\in M,
\]
we have
\begin{align}
\Delta f(x)=-\sum_{k=1}^{7-n}\left(  \langle S_{\eta},S_{\eta_{k}} \rangle
_{x}+n\delta_{1k}\right)  \langle\Gamma_{x}(\eta_{k}),v \rangle,
\end{align}
where $\{\eta_{1}=\eta,\cdots, \eta_{7-n} \}$ is any orthonormal frame in a
neighborhood of $x$ in the normal bundle of $M.$
\end{lem}

\begin{proof}
Let $x\in M$ be given. Let $\{E_{1},\cdots, E_{n}\}$ be an orthonormal frame
in a neighbourhood of $x,$ geodesic at $x\in M$ and that diagonalizes the
second fundamental form $S_{\eta}$ at $x\in M.$ Then,
\begin{align}
\label{1}\Delta\langle\Gamma(\eta),v \rangle=\sum_{i=1}^{n} E_{i}E_{i}%
\langle\Gamma(\eta),v \rangle=\sum_{i=1}^{n} E_{i}E_{i}\langle\eta,V \rangle,
\end{align}
where $V(x)=x\cdot v.$ Since $V$ is a Killing vector field of $\mathbb{S}^{7}$
it follows that, at $x,$
\begin{align*}
\Delta f  &  =\sum_{i=1}^{n}\left(  \langle\nabla_{E_{i}}\nabla_{E_{i}}\eta,V
\rangle+2\langle\nabla_{E_{i}}\eta,\nabla_{E_{i}}V \rangle+\langle\eta
,\nabla_{E_{i}}\nabla_{E_{i}}V \rangle\right) \\
&  = \sum_{i=1}^{n}\left(  \langle\nabla_{E_{i}}\nabla_{E_{i}}\eta,V
\rangle+\langle\eta,\nabla_{E_{i}}\nabla_{E_{i}}V \rangle\right)  .
\end{align*}

We claim that
\[
\sum_{i=1}^{n}\langle\eta,\nabla_{E_{i}}\nabla_{E_{i}}V \rangle=-nf.
\]
Indeed,
\begin{align*}
\sum_{i=1}^{n}\langle\eta, \nabla_{E_{i}}\nabla_{E_{i}}V \rangle &  =
\sum_{i=1}^{n}\left(  E_{i}\langle\eta,\nabla_{E_{i}}V \rangle-\langle
\nabla_{E_{i}}\eta,\nabla_{E_{i}}V \rangle\right) \\
&  = -\sum_{i=1}^{n} E_{i}\langle E_{i},\nabla_{\eta}V \rangle\\
&  = -\sum_{i=1}^{n} \langle\nabla_{E_{i}}E_{i},\nabla_{\eta}V \rangle
-\sum_{i=1}^{n} \langle E_{i},\nabla_{E_{i}}\nabla_{\eta}V \rangle\\
&  = -\sum_{i=1}^{n} \langle E_{i},\nabla_{E_{i}}\nabla_{\eta}V \rangle.
\end{align*}
That is,
\begin{align}
\label{eq2}\sum_{i=1}^{n}\langle\eta, \nabla_{E_{i}}\nabla_{E_{i}}V \rangle &
= -\sum_{i=1}^{n} \langle E_{i},\nabla_{E_{i}}\nabla_{\eta}V \rangle.
\end{align}
Extending each $E_{i}$ to a neighborhood of $x$ in $\mathbb{S}^{7}$ such that
the extension is parallel along the geodesic given by $\eta$ we obtain
\begin{align}
\label{eq3}0=\eta\langle\nabla_{E_{i}}V,E_{i} \rangle=\langle\nabla_{\eta
}\nabla_{E_{i}}V,E_{i} \rangle.
\end{align}

By (\ref{eq2}) and (\ref{eq3}), and using the curvature tensor of
$\mathbb{S}^{7},$
\begin{align*}
nf  &  =\sum_{i=1}^{n}\left(  \langle\nabla_{E_{i}}\nabla_{\eta}V,E_{i}%
\rangle-\langle\nabla_{\eta}\nabla_{E_{i}}V,E_{i}\rangle+\langle
\nabla_{\lbrack\eta,E_{i}]}V,E_{i}\rangle\right) \\
&  =-\sum_{i=1}^{n}\left(  \langle\nabla_{E_{i}}\nabla_{E_{i}}V,\eta
\rangle+\langle\nabla_{E_{i}}V,[\eta,E_{i}]\rangle\right) \\
&  =-\sum_{i=1}^{n}\langle\nabla_{E_{i}}\nabla_{E_{i}}V,\eta\rangle.
\end{align*}


Therefore, substituting in (\ref{1}) we have, at $x,$
\begin{align*}
\Delta f  &  = -nf+\sum_{i=1}^{n}\langle\nabla_{E_{i}}\nabla_{E_{i}}\eta,V
\rangle.
\end{align*}
Let $\{\eta_{1}=\eta,\ \cdots,\ \eta_{7-n} \}$ be an orthonormal frame in a
neighborhood of $x$ in the normal bundle of $M.$ Writing $V$ in terms of the
tangent and normal frames,
\[
V=\sum_{i=1}^{n} v_{i} E_{i}+\sum_{k=1}^{7-n}f_{k}\eta_{k},
\]
we have
\begin{align*}
\sum_{i=1}^{n}\langle\nabla_{E_{i}}\nabla_{E_{i}}\eta,V \rangle &
=\sum_{i,j=1}^{n}v_{j}\langle\nabla_{E_{i}}\nabla_{E_{i}}\eta,E_{j}
\rangle+\sum_{k=1}^{7-n}\sum_{i=1}^{n}\langle\nabla_{E_{i}}\nabla_{E_{i}}%
\eta,\eta_{k} \rangle f_{k}\\
&  =\sum_{i,j=1}^{n}v_{j}\langle\nabla_{E_{i}}\nabla_{E_{i}}\eta,E_{j}
\rangle-\sum_{k=1}^{7-n}\langle S_{\eta},S_{\eta_{k}} \rangle f_{k}.
\end{align*}

We now show that
\[
\displaystyle{\sum_{i=1}^{n}\langle\nabla_{E_{i}}\nabla_{E_{i}}\eta,E_{j}
\rangle}=0,
\]
which proves the lemma. First, we note that, at $x,$
\[
[E_{i},E_{j}]=[E_{i},E_{j}]^{\top}=(\nabla_{E_{i}}E_{j}-\nabla_{E_{j}}%
E_{i})^{\top}=0,
\]
where $\top$ denotes the orthogonal projection on $TM$. Then, using the
curvature tensor of $\mathbb{S}^{7}$ and the equality $\langle[E_{j}%
,E_{i}],\eta\rangle=0$ along $M,$
\begin{align*}
\langle\nabla_{E_{j}}\nabla_{E_{i}}E_{i},\eta\rangle &  = \langle\nabla
_{E_{i}}\nabla_{E_{j}}E_{i},\eta\rangle\\
&  =\langle\nabla_{E_{i}}[E_{j},E_{i}],\eta\rangle-\langle\nabla_{E_{i}}%
\nabla_{E_{i}}E_{j},\eta\rangle\\
&  =-\langle\nabla_{E_{i}}\nabla_{E_{i}}E_{j},\eta\rangle.
\end{align*}

Therefore,
\begin{align*}
\sum_{i=1}^{n}\langle\nabla_{E_{i}}\nabla_{E_{i}}\eta,E_{j} \rangle &
=-\sum_{i=1}^{n}\langle\nabla_{E_{i}}\nabla_{E_{i}}E_{j},\eta\rangle\\
&  = \sum_{i=1}^{n}\langle\nabla_{E_{j}}\nabla_{E_{i}}E_{i},\eta\rangle\\
&  = \sum_{i=1}^{n} \left(  E_{j}\langle\nabla_{E_{i}}E_{i},\eta
\rangle-\langle\nabla_{E_{i}}E_{i},\nabla_{E_{j}}\eta\rangle\right) \\
&  = 0
\end{align*}
concluding with the proof of the lemma.
\end{proof}

\begin{proof}
[Proof of the Theorem \ref{mink}]%
We prove simultaneously the equivalences among
(i), (ii) and (iii). Let $\{v_{1},...,v_{7}\}$ be an orthonormal basis of the
tangent space $T_{1}\mathbb{S}^{7}.$ Fix $x\in M$ and let $\eta$ be a unit normal section parallel in the normal connection of $M.$ Consider an orthonormal frame $\{\eta_1=\eta,\ \eta_2,\ \cdots,\ \eta_k \}$ on a neighborhood of the normal bundle of $M$ at $x.$ Setting $$\gamma_{\eta_j}(x):=\Gamma_x(\eta_j(x))=x^{-1}\cdot \eta_j(x),$$  $j=1,\ \cdots,\ k,$ we have from the Lemma \ref{laplacian}
\begin{align*}
\Delta\gamma_{\eta}(x)  &  =\sum_{i=1}^{7}(\Delta\langle\gamma_{\eta},v_{i} \rangle)v_{i}\\
&  = -\sum_{i=1}^{7}\left(  \sum_{j=1}^{k}\left(  \langle S_{\eta},S_{\eta_{j}} \rangle+(7-k)\delta_{1j}\right)  \langle\gamma_{\eta_j},v_{i}
\rangle\right)  v_{i}\\
&  = -\sum_{j=1}^{k}\left(  \langle S_{\eta},S_{\eta_{j}} \rangle
_{x}+(7-k)\delta_{1j}\right)  \gamma_{\eta_{j}}(x).
\end{align*}
Since $\gamma_{\eta}$ is harmonic if and only if $\Delta\gamma_{\eta}$
is a multiple of $\gamma_{\eta}$ and since $\gamma_{\eta_{1}}\dots
\gamma_{\eta_{k}}$ are linearly independent, we obtain that $\gamma_{\eta}$ is a harmonic map if and only if
\begin{align*}
\Delta\gamma_{\eta}(x)  &  = -\left(  \Vert S_{\eta} \Vert
^{2}+(7-k)\right)  \gamma_{\eta}(x),
\end{align*}
and the last equality holds if and only if $\forma{ S_\eta,S_{\eta_{j}}}=0$ for $j\neq 1$, that is, $\forma{\mathcal{B}^{\ast}\mathcal{B}(\eta),\eta_j}=0.$  The last equality is equivalent to $\eta$ be an eigenvalue of $\mathcal{B}^{\ast}\mathcal{B}$ with eigenvalue $$\forma{\mathcal{B}^{\ast}\mathcal{B}(\eta),\eta}=\forma{\mathcal{B}(\eta),\mathcal{B}(\eta)}=\norma{S_\eta}^2.$$
\end{proof}

\begin{proof}
[Proof of the Theorem \ref{isoparametric}] Consider $x\in M$ and the linear operator $\mathcal{B}^{\ast}\mathcal{B}$ at $x.$ That is, $\mathcal{B}^{\ast}\mathcal{B}(x): T_x^\perp M \to T_x^\perp M,$ where $\mathcal{B}(x)(\eta)=S_{\eta(x)}.$  Since   $\mathcal{B}^{\ast}\mathcal{B}(x)$  is non-negative and self-adjoint there is an orthonormal basis  $\{\nu_1,\ \cdots,\ \nu_k \}\in T_xM^\perp$ of eigenvectors of  $\mathcal{B}^{\ast}\mathcal{B}(x)$  with eigenvalues $0\leq \sigma_1\leq\cdots\leq \sigma_k.$

On the other hand, it is well known that submanifold of the sphere $\mathbb{S}^n$ is isoparametric in $\mathbb{S}^n$ if and only it is isoparametric in $\mathbb{R}^{n+1}.$ From  \cite{T} there are parallel orthonomal  unit normal sections $\{\tau_1,\ \cdots,\ \tau_{k+1} \}
$ of $M$ in $\mathbb{R}^{n+1}$. Define, for $1\leq j\leq k,$ $$\eta_j(y)=\sum_{i=1}^{k+1}a_{ji}\tau_i(y), \  y\in M,$$ if  $$\nu_j=\sum_{i=1}^{k+1}a_{ji}\tau_i(x).$$ The vector fiels $\eta_j$ are orthogonal to $M$ in $\mathbb{S}^n$ since the vector field $V(y)=y,$ $y\in M,$ is parallel in the normal connection of $M$ in $\mathbb{R}^{n+1}$ and $$\forma{\eta_j(x),V(x)}=\forma{\nu_j,V(x)}=0.$$ Moreover, $\eta_j$ is  parallel since it is a linear combination, with constant coefficients, of parallel vector fields,  $1\leq j\leq k.$  By the definition of isoparametric submanifolds it follows that the eigenvalues of each $S_{\eta_j}$ are constant so that   $\eta_j $ is an eigenvector of $\mathcal{B}^{\ast}\mathcal{B}$ with the constant eigenvalue  $\sigma_j=\norma{S_{\eta_j}}^2,$ $1\leq j\leq k.$ 
\end{proof}

\begin{proof}
[Proof of the Theorem \ref{image}]%
Let $\eta$ be an unit normal section parallel in the normal connection of $M$. Assume that the image of
$\gamma_{\eta}$ is contained in an open hemisphere of $\mathbb{S}^{6}$ centered
at a vector $v$. Then $\langle\gamma_{\eta}(x), v \rangle>0$ for all $x\in M.$
Since $M$ is compact there is a neighbourhood $U$ of $v$ in $\mathbb{S}^{6}$
such that $\langle\gamma_{\eta}(x), w \rangle>0$ for all $x\in M$ and for all
$w\in U.$ Clearly, in $U$ we may choose $7$ linearly independent vectors
$w_{1},\dots, w_{7}.$ From the equality (i) of Theorem \ref{mink}
it follows that each function $f_{i}=\langle\gamma_{\eta},w_{i}
\rangle$ is superharmonic, $1\leq i\leq7.$ Since $M$ is compact $f_{i}$ is
constant and then, $\Delta f_{i}=0$. Since the coefficient of $\gamma_{\eta}$
in (i) is nonzero we obtain also from (i) that $f_{i}$ is identically zero.
We then conclude that at each point $x\in M$ the nonzero vector $\gamma_{\eta
}(x)$ is orthogonal to $7$ linearly independent vectors in a $7-$dimensional
vector space, contradiction! This proves the theorem.
\end{proof}


\begin{thebibliography}{9}                                                                                                %


\bibitem {B}J. Baez: \emph{The octonions}, Bull. of the Amer. Math. Soc., Vol.
39, n. 2, 145--205, 2001.



\bibitem {BFLR}F. Bittencourt, P. Fusieger, E. R. Longa, J. Ripoll: \emph{
Normed division algebras, Gauss map and the topology of constant mean
curvature hypersurfaces $\mathbb{S}^{7}$ and $\mathbb{CP}^{3}$}, preprint, arXiv:1703.02560.

\bibitem{BN} A. A. Borisenko and Yu A. Nikolaevskii: \emph{Grassmann manifolds and the Grassmann image
of submanifolds}  Russ. Math. Surv. 46 45, 1991.


\bibitem{CL} S.-S. Chern and R. K. Lashof: \emph{On the total curvature of immersed manifolds}. Amer. J. Math., 
79, 306-318, 1957.


\bibitem {DG} E. De Giorgi: \emph{Una estensione del teorema di Bernstein},
Ann. Scuola Norm., Sup. Pisa (3) 19, 79--85, 1965


\bibitem {EL}J. Eells and L. Lemaire: \emph{Selected topics in harmonic maps},
C.B.M.S. Regional Conference Series in Mathematics 50 (American Mathematical
Society, Providence, RI, 1983)




\bibitem{HO}
 D. A. Hoffman and R. Osserman: \emph{The Gauss map of surfaces in $\mathbf{R}^n$}, J. Differential Geom. 18, no. 4, 733--754, 1983. 

\bibitem {I} T. Ishihara: \emph{The Harmonic Gauss maps in a generalized
sense}, J. London Math. Soc. 26 104-112, 1982.



\bibitem{O} Morio Obata: \emph{ The Gauss map of immersions of Riemannian manifolds in
spaces of constant curvature}, J. Differential Geometry, 2, 217-–223, 1968 .

\bibitem{Os} Robert Osserman: \emph{Minimal surfaces, Gauss maps, total curvature, eigenvalue
estimates, and stability}. In The Chern Symposium 1979 (Proc. Internat.
Sympos., Berkeley, Calif., 1979), Springer, New York, 199-–227, 1980.

\bibitem {PT} R. Palais, C. Terng: \emph{Critical point theory and submanifold
geometry}, Lecture Notes in Mathematics, 1353, 1980.

\bibitem {QT}C. Qian, Z. Tang: \emph{Isoparametric foliations, a problem of
Eells--Lemaire and conjectures of Leung}, Proc. London Math. Soc. 3, 112,
 979--1001, 2016.





\bibitem {JS}J. Simons: \emph{Minimal varieties in riemannian manifolds},
Ann. of Math. (2) 88, 62--105, 1968.

\bibitem {T}C. Terng: \emph{Isoparametric submanifolds and their Coxeter
groups, }J. of Differential Geometry, 21, 79--107, 1985.
\end{thebibliography}
\end{document}